\documentclass{cimart}

\usepackage{enumerate}
\hypersetup{hypertexnames=false}

\theoremstyle{cmplain}
\newtheorem{claim}[definition]{Claim}

\newcommand{\charac}[1]{\mathrm{char}\left(#1\right)}
\newcommand{\etale}{\'{e}tale}
\newcommand{\set}[1]{\{#1\}}
\newcommand{\suchthat}{\,:\ \,}

\DeclareMathOperator{\tr}{Tr}
\DeclareMathOperator{\norm}{Norm}

\title{General polynomials and eigenvalues over Cayley--Dickson algebras}

\authors{Adam Chapman, Ilan Levin and Solomon Vishkautsan}

\authorinfo[A. Chapman]{Academic College of Tel-Aviv-Yaffo, Israel}{adam1chapman@yahoo.com}

\authorinfo[I. Levin]{Bar-Ilan University, Israel}{ilan7362@gmail.com}

\authorinfo[S. Vishkautsan]{Tel-Hai University of Kiryat Shmona in the Galilee, Israel}{wishcow@gmail.com}

\abstract{%
    In this paper, we provide an explicit method for determining the zero set of monic quadratic general polynomials over Cayley--Dickson algebras with any base field of characteristic not equal to $2$, which reduces the problem to finding the roots of univariate polynomials over the base field. For the special case of locally-complex Cayley--Dickson algebras over the reals, we prove that every such polynomial has a root. We prove the existence of right eigenvalues for any $2 \times 2$ matrix over Cayley's octonions (the real octonion division algebra), and our root-finding method allows the computation of some of the right eigenvalues.
    }

\keywords{
    Cayley--Dickson algebras, General polynomials, Monic quadratic polynomials, Right eigenvalues, Octonion algebras.
    }

\msc{
    17A75 (primary); 17A45, 17A35, 17D05, 12D10 (secondary).
    }

\VOLUME{35}
\YEAR{2027}
\ISSUE{2}
\NUMBER{4}
\DOI{https://doi.org/10.46298/cm.18874}
\licence{CC BY-SA 4.0}
\editinfo{July 28, 2026}{October 9, 2026}{Ivan Kaygorodov and Maxime Fairon}
\acknowledgments{The authors are grateful to the anonymous referees for a careful reading of the manuscript and for many corrections and suggestions that substantially improved the paper.}

\begin{document}

\section{Introduction}

The study of polynomials over non-commutative and non-associative algebras requires distinguishing between standard left polynomials and general polynomials. In the standard ring of left polynomials, the variable is treated as central with respect to the coefficients (it commutes and associates with elements of the algebra). For associative division rings, the theory of left polynomials and their roots is
classical; see \cite[\S16]{Lam:2001} for a textbook account. In this paper, however, we look instead at general polynomials, for which the variable does not commute with the coefficients (see Section~\ref{sec:poly} for the definitions of these polynomial rings). For example, the general polynomials $(ij)x, i(xj)$ and $(ix)j$ are distinct over an octonion algebra (where $i,j,\ell$ denote the standard generators of an octonion algebra). General polynomials were first formally studied by Gordon and Motzkin \cite{GordonMotzkin1965} for division algebras, and were later generalized to other algebras by R\"ohrl \cite{Rohrl:1978} and Wilczy{\'n}ski \cite{Wilczynski:2014}.

A classical question regarding univariate polynomial equations is the existence of roots. Over the complex field $\mathbb{C}$ this is of course the Fundamental Theorem of Algebra. This was generalized to general polynomials over Hamilton's quaternions $\mathbb{H}$ by Eilenberg and Niven \cite{eilenberg-niven:1944}, and to Cayley's octonions $\mathbb{O}$ by Jou \cite{jou1950}. Both generalizations impose the extra condition that there is a unique monomial of highest degree. Some polynomials which do not satisfy this condition, like $f(x) = ix - xi + 1$, have no roots (in either algebra). The most general theorem so far regarding the existence of roots is due to Wilczy{\'n}ski \cite{Wilczynski:2014}, who proved the existence of roots for a larger set of polynomials.

A natural generalization of the existence problem of roots for general polynomials over the quaternions and the octonions is to consider general polynomials over Cayley--Dickson algebras, as this class contains both (see Section~\ref{sec:cd} for definition and properties). In this paper, we establish an explicit method for determining the zero set of monic quadratic general polynomials over Cayley--Dickson algebras with any base field of characteristic not equal to 2. The method is explicit in the sense that it reduces the problem to finding the roots in the base field of univariate polynomials of degree at most twice the dimension of the algebra, together with linear algebra. Furthermore, for the special case of locally-complex Cayley--Dickson algebras, we prove that every such polynomial has a root. 

Our approach follows Xu, Feng and Yao \cite{Xu2014Zeros}, who described the zero set of monic quadratic general polynomials over $\mathbb{H}$ and showed that such polynomials always have a zero. We extend their results in three directions: to arbitrary base fields of characteristic not $2$; to Cayley--Dickson algebras $A_n$ of arbitrary dimension, where for $n\ge 4$ the existence of zeros in the locally-complex case does not follow from the theorems of Eilenberg--Niven and Jou; and to linear terms with arbitrary bracketing, as needed in the non-associative case.

A primary motivation for studying roots of general polynomials over Cayley--Dickson algebras stems from the theory of matrix eigenvalues. Over the complex field, eigenvalues are roots of the characteristic polynomial. In the non-commutative and non-associative setting, left eigenvalues are distinct from right eigenvalues (see Section~\ref{sec:eigenvalues} for the definitions). 
Finding right eigenvalues is particularly important when discussing the diagonalization of matrices, i.e., for solving matrix equations of the form $AU=UD$, where $D$ is diagonal. 
In some cases, left and right eigenvalues over quaternion and octonion algebras have been shown to be roots of certain left and general polynomials. 

The right eigenvalues of matrices over Hamilton's quaternions were fully described by Lee~\cite{Lee:1949}. The paper \cite{ChapmanMachen:2017} generalized Lee's results to any associative central division algebra. Left eigenvalues are considerably more difficult, and in fact, only the case of $M_2(\mathbb{H})$ was fully solved, see \cite{HuangSo:2001} (and \cite{HuangSo2002}). The existence of left eigenvalues for matrices in $M_n(\mathbb{H})$ was proven by Wood \cite{wood1985}.

Previous work on eigenvalues of octonion matrices includes papers related to the computation of real and non-real left and right eigenvalues of $3\times 3$ Hermitian matrices over Cayley's octonions (see \cite{dray-manogue1998,DrayManogue:2015}). Recent work by two of the authors (see \cite{ChapmanVishkautsan:2025}) has shown that every $2 \times 2$ matrix over Cayley's octonions admits a left eigenvalue. 

We prove the existence of right eigenvalues for arbitrary $2 \times 2$ matrices over Cayley's octonions by reducing the problem to the existence of solutions to a monic quadratic general equation. Thus our method for computing the zero set of monic quadratic polynomials allows some of the eigenvalues to be computed explicitly (see Section~\ref{sec:right-eigenvalues}).

   We finish the paper by computing the right spectrum of triangular matrices over octonion division algebras and the left spectrum of triangular matrices over arbitrary division algebras (see Section~\ref{sec:triangular}). For the case of left eigenvalues, this corrects a proof given in \cite{ChapmanVishkautsan:2025}.

\section{Cayley--Dickson algebras} \label{sec:cd}

Given an algebra $A$ with an involution $a \mapsto \bar{a}$ over a field $F$ 
and $\gamma \in F^\times = F \setminus \{ 0 \}$, the Cayley--Dickson doubling process produces a new algebra $B=A \oplus A\ell$, denoted also by $A\{\gamma\}$, where $\ell$ is an indeterminate. The operations in $B$ are defined  by the rules 
\begin{align*}
\alpha (q+r\ell) & = \alpha q + \alpha r\ell, \\
(q+r\ell)+(s+t\ell) & = (q+s)+(r+t)\ell \ \text{and}\\
(q+r\ell)(s+t\ell) & = qs+\bar{t}r \gamma+(tq+r\bar{s})\ell    
\end{align*}
 for any $q,r,s,t\in A,\: \alpha \in F$. The involution then extends to $B$ by $\overline{q+r\ell}=\bar{q}-r\ell$. From these definitions, it follows that $\ell^2 = \gamma$. 

For $\charac{F}\ne 2$, set $A_0=F$, with $a \mapsto \bar{a}$ being the identity map, and for $n \ge 1$ define $A_n = A_{n-1}\{\gamma_n\}$ for a choice of $\gamma_n \in F^\times$, writing $\ell_n$ for the element $\ell$ adjoined at this step, so that $\ell_n^2=\gamma_n$. Thus $A_1 = F \oplus F\ell_1 \cong F[z]/(z^2-\gamma_1)$ is a quadratic \etale{} algebra, $A_2 = A_1\{\gamma_2\}$ is a quaternion algebra, $A_3 = A_2\{\gamma_3\}$ is an octonion algebra, and $A_4 = A_3\{\gamma_4\}$ is known as a sedenion algebra. Continuing in this way, one obtains an infinite sequence $\{A_n\}_{n\ge 0}$ of algebras called Cayley--Dickson algebras.

The sequence of Cayley--Dickson doublings is characterized by a progressive loss of algebraic structure. Commutativity fails at $A_2$, associativity at $A_3$, though alternativity holds, and alternativity is lost at $A_4$. However, all algebras in the sequence retain flexibility and power-associativity (see \cite{Schafer:1954}).

We define the standard basis $E_n = \{ e_m^{(n)} \; | \; m = 0 , \dots, 2^n - 1\}$ of $A_n$ inductively:
\begin{enumerate}[(i)]
    \item $e_0^{(0)} = 1$;
    \item $e_m^{(n+1)} = \begin{cases}
    e_m^{(n)}, & 0 \leq m \leq 2^n - 1,\\
    e_{m - 2^n}^{(n)}\ell_{n+1}, & 2^n \leq m \leq 2^{n+1} - 1,
    \end{cases}
    \quad n \in \mathbb{N}_{0}$,
\end{enumerate}
(with  $\ell_n$ the indeterminate used to define $A_n$). It is easily seen that $e_0^{(n)}=1$ for any $n\ge 0$. The elements of $E_n \setminus \{ 1 \}$ anticommute pairwise, that is, $e_k^{(n)} e_m^{(n)} = - e_m^{(n)} e_k^{(n)}$ for $1 \leq k < m \leq 2^n - 1$.

All Cayley--Dickson algebras are equipped with a linear trace map $\tr(x)=x+\bar{x}$ and quadratic norm map $\norm(x)= x \cdot \bar{x}$ with values in the base field $F$. Any element $a$ in a Cayley--Dickson algebra satisfies the quadratic equation $a^2-\tr(a)a+\norm(a)=0$. Note that $\tr(\ell)=0$ and $\norm(\ell)=-\gamma$, in the notation from the beginning of the section.

A real unital algebra is called locally-complex if any non-scalar element generates a subalgebra isomorphic to $\mathbb{C}$ (see \cite{brevsar2011locally}). In \cite{Chapman-et-al:2023} it is proven that locally-complex Cayley--Dickson algebras are exactly the algebras (up to isomorphism) in the classical Cayley--Dickson doubling sequence generated from $F=\mathbb{R}$ by setting $\gamma_n = -1$ for all $n\ge 1$. The first few algebras in this sequence are $A_1=\mathbb{C}$, the complex field; $A_2 = \mathbb{H}$, Hamilton's quaternions; $A_3 = \mathbb{O}$, Cayley's octonions; $A_4 = \mathbb{S}$, the real sedenions.

\section{Cayley--Dickson polynomials} \label{sec:poly}

When working with polynomials whose coefficients lie in Cayley--Dickson algebras, there are two common types of constructions, left polynomials and general polynomials. We recall the constructions of these polynomial algebras below. 

\subsection{Left polynomials}

Given a Cayley--Dickson algebra $A$ over a field $F$, the ring of \emph{left} (or \emph{standard}) polynomials $A[x]$ is defined to be 
\begin{equation} \label{eq:left-polynomials}
    A[x] = A \otimes_F F[x],
\end{equation}
i.e., the scalar extension of $A$ to the polynomial ring in one variable over $F$.
The variable $x$ is thus central in $A[x]$. Each polynomial $f(x) \in A[x]$ can therefore be written with the powers of $x$ placed on the right-hand side and the coefficients placed on the left-hand side (which is why they are referred to as left polynomials), $f(x)=\sum_{i=0}^m c_i x^i$.

For any $a \in A$, we define the substitution map $\Psi_a: A[x]\to A$ by $\Psi_a(f)=\sum_{i=0}^m c_i a^i$, and we denote $f(a)=\Psi_a(f)$ as usual. Note that $\Psi_a$ is in general not a ring homomorphism (see \cite[\S16]{Lam:2001}).

One of the benefits of left polynomials is that for Cayley--Dickson \emph{division} algebras of dimension at most $8$, we have that $f(x)\in A[x]$ decomposes as $g(x)(x-a)$ if and only if $f(a)=0$, i.e., $(x-a)$ is a linear right factor of $f(x)$ if and only if $a$ is a root of $f(x)$ (see \cite{Wedderburn:1921,GordonMotzkin1965} and \cite[Proposition~16.2]{Lam:2001} for associative division algebras, and \cite{Chapman:2020b} for octonion algebras). This property fails for Cayley--Dickson algebras of dimension $16$ and above (see \cite{ChapmanVishkautsan:2025}).

\subsection{General polynomials}

Given a Cayley--Dickson algebra $A$ over a field $F$, we want to define polynomials whose variable does not behave as central with respect to the coefficients from $A$. This was first formulated by Gordon and Motzkin in \cite{GordonMotzkin1965} for division algebras $A$, and was generalized by R\"ohrl \cite{Rohrl:1978} and Wilczy{\'n}ski \cite{Wilczynski:2014} to general, not necessarily associative $F$-algebras. We follow the latter construction.

Let $F\langle(\!(x)\!)\rangle$ be the free unital non-associative $F$-algebra on one generator $x$. It has an $F$-basis consisting of $1$ together with all non-associative words in $x$. The \emph{formal polynomial algebra} over $A$ is the coproduct (free product) in the category of unital $F$-algebras:
\begin{equation}\label{eq:formal-polynomials}
    A\langle(\!(x)\!)\rangle = A * F\langle(\!(x)\!)\rangle;
\end{equation}
see \cite[Sections~2--3]{Wilczynski:2014} for an explicit construction (Compare \eqref{eq:formal-polynomials} with \eqref{eq:left-polynomials}). The algebra $A\langle(\!(x)\!)\rangle$ carries a natural $\mathbb{N}$-grading by degree in $x$, where $\deg(x) = 1$ and $\deg(a) = 0$ for all $a \in A$. As usual, a polynomial $f \in A\langle(\!(x)\!)\rangle$ is \emph{quadratic} if it has degree $2$.

By the universal property of the coproduct \cite[Theorem~3.1]{Wilczynski:2014}, every $q \in A$ defines a unique \emph{evaluation homomorphism} $\operatorname{ev}_q \colon A\langle(\!(x)\!)\rangle \to A$ restricting to the identity on $A$ and sending $x \mapsto q$. We denote $\operatorname{ev}_q(f)$ by $f(q)$.

The algebra $A\langle(\!(x)\!)\rangle$ does not inherit the polynomial identities of $A$; for instance, even if $A$ is alternative or flexible, $A\langle(\!(x)\!)\rangle$ satisfies neither property. One can pass to a quotient satisfying the identities of $A$ by factoring out the ideal generated by all evaluations of these identities at elements of $A\langle(\!(x)\!)\rangle$ (cf.\ \cite{GordonMotzkin1965,Wilczynski:2014}); for example, when $A$ is an octonion algebra, factoring out the ideal generated by all alternators $(u,u,v)$ and $(v,u,u)$ such that $u,v\in A\langle(\!(x)\!)\rangle$, where $(u,v,w)=(uv)w-u(vw)$ denotes the associator, yields an alternative algebra. However, since we are only concerned with evaluation at elements of $A$, a formal polynomial and its image in any such quotient define the same function on $A$. We therefore work directly in $A\langle(\!(x)\!)\rangle$, refer to its elements as \emph{general polynomials}, and apply the identities of $A$ only upon evaluation.

\section{Monic quadratic general polynomials over Cayley--Dickson algebras}
\label{sec:monic}

    This section generalizes the approach of \cite{Xu2014Zeros} for real quaternions to Cayley--Dickson algebras of arbitrary dimension over an arbitrary field of characteristic not $2$; in particular, our proof of Theorem~\ref{thm:real-quadratic} follows the second proof of \cite[Theorem~2.11]{Xu2014Zeros}.
    
Let $F$ be a field of characteristic not 2. Let $n\ge 1$ and let $A_n$ be the $2^n$-dimensional $F$-algebra formed by iterating the Cayley--Dickson construction $n$ times starting with the base field $F$. Set $e_0, e_1, \dots, e_m$ to be the standard basis of $A_n$, with $e_0 = 1$, where $m = 2^n-1$. Recall that the basis elements satisfy $e_ie_j = -e_je_i$ for any nonzero $i \neq j$. Denote $e_i^2 = \alpha_i \in F$ for $1\le i \le m$.

We are interested in solving monic quadratic general polynomials over $A_n$. Namely, polynomials $f(x)$ that have one quadratic term $x^2$, a finite number of linear terms, and a constant term. A linear term in this context is a product of elements from $A_n$ and precisely one $x$, in any order and any positioning of brackets. For example, the polynomial $f(x)=x^2+(e_1((xe_2)e_4))e_1+34(e_1+e_2)x+8$ is a monic quadratic general polynomial over $\mathbb{O}$, Cayley's octonions over the reals.

Let 
\begin{equation} \label{eq:monic-quadratic}
f(x) = x^2 + f_1(x) + r    
\end{equation}
be a monic quadratic general polynomial, where $r=\sum_{i=0}^{m} r_ie_i \in A_n$ and $f_1(x)$ is the linear part of $f$. To determine the zeros of $f$, we
evaluate it at an element
\begin{equation} \label{eq:el-basis-expansion}
q = t + \sum_{i=1}^{m} q_i e_i \in A_n, \qquad t,q_1,\dots,q_m \in F,
\end{equation}
so that $f(q) = \operatorname{ev}_q(f) = q^2 + f_1(q) + r$, and we treat
$t,q_1,\dots,q_m$ as unknowns in $F$. Note that $t = \tr(q)/2$.

Since every linear term of $f_1$ contains exactly one occurrence of $x$ and
multiplication in $A_n$ is $F$-bilinear, the map 
$q \mapsto f_1(q)$ is $F$-linear. For $i = 0,\dots,m$ let
$\pi_i \colon A_n \to F$ denote the $i$-th coordinate function with respect
to the standard basis, and define
\[
G_i \colon A_n \to F, \qquad G_i(q) = \pi_i\bigl(f_1(q)\bigr),
\]
so that
\begin{equation} \label{eq:linear-part-decomp}
f_1(q) = \sum_{i=0}^{m} G_i(q)\, e_i \qquad \text{for all } q \in A_n.
\end{equation}
Each $G_i$ is $F$-linear, so identifying $q$ with its coordinate vector
$(t,q_1,\dots,q_m) \in F^{m+1}$ we may write
\[
G_i(t,q_1,\dots,q_m) = g_{i,0}t + g_{i,1}q_1 + \dots + g_{i,m}q_m,
\]
where  $g_{i,k} = \pi_i\bigl(f_1(e_k)\bigr) \in F$; that is, $(g_{i,k})_{i,k}$ is the matrix of the map $q \mapsto f_1(q)$ in the
standard basis. Finally, since $e_1,\dots,e_m$ pairwise anticommute and
$e_i^2 = \alpha_i$, we have
\[
q^2 = t^2 + \sum_{i=1}^{m} \alpha_i q_i^2 + 2t \sum_{i=1}^{m} q_i e_i .
\]
Comparing the coefficients of $e_0,\dots,e_m$ in $f(q) = 0$, we obtain the
following system of equations in the unknowns $t,q_1,\dots,q_m$:
\begin{align*}
t^2+\alpha_1 q_1^2 +\dots+ \alpha_m q_m^2 + G_0(t,q_1,\dots,q_{m})+r_0 &= 0 \tag{$E_0$} \label{eq:E_0}\\
2tq_1 + G_1(t,q_1,\dots,q_{m})+r_1 &= 0 \tag{$E_1$} \label{eq:E_1}\\
&\ \vdots \\
2tq_m + G_m(t,q_1,\dots,q_{m})+r_m &= 0 \tag{$E_m$} \label{eq:E_m}
\end{align*}

Now think of $t$ as fixed, in which case equations \eqref{eq:E_1}, \ldots, \eqref{eq:E_m} boil down to the single matrix equation over $F$:
\begin{align*}
&\qquad
\begin{pmatrix}
2t+g_{1,1} & g_{1,2} & \dots & g_{1,m} \\
g_{2,1} & 2t+g_{2,2} & \dots & g_{2,m} \\
\vdots  & \vdots  & \ddots & \vdots \\
g_{m,1} & g_{m,2} & \dots & 2t+g_{m,m}
\end{pmatrix}
\cdot \begin{pmatrix}
q_1 \\
q_2 \\
\vdots \\
q_m
\end{pmatrix}
= \begin{pmatrix}
-r_1-g_{1,0}t \\
-r_2-g_{2,0}t \\
\vdots \\
-r_m-g_{m,0}t
\end{pmatrix} \tag{E} \label{eq:E}
\end{align*}

Denote by $M(t)$ the $m\times m$ matrix in \eqref{eq:E}, by $\vec{r}(t)$ its right-hand side, and write $\vec q=(q_1,\dots,q_m)^T$, so that \eqref{eq:E} reads $M(t)\vec q=\vec{r}(t)$. By Cramer's rule, when $\operatorname{det}(M(t)) \neq 0$, the unique solution vector to \eqref{eq:E} is given by $q_i(t) = \frac{\operatorname{det}(M_i(t))}{\operatorname{det}(M(t))}$, where $M_i(t)$ is the matrix obtained by substituting $\vec{r}(t)$ in the $i$-th column of $M(t)$. Denote the determinants of $M(t),M_i(t)$ by $D(t), D_i(t)$, respectively. 

Notice that the set $S := \{t \in F \suchthat D(t) = 0 \}$ is finite and $|S| \leq m$. Of the values in $S$, we only care about the subset $S' := S \cap \set{t \suchthat M(t)\vec{q} = \vec{r}(t) \text{ has a solution}}$. As we will see, the values in $S'$
will help us determine the geometric structure of the zero set of \eqref{eq:monic-quadratic}.
Before proving a lemma about the number of possible trace values of the zero set, we prove the following useful claim:

\begin{claim} \label{compatible}
    If $t_0 \in S'$, then $t-t_0$ divides $D(t)$ and $D_i(t)$ for all $1 \leq i \leq m$. Consequently, the polynomials 
    \[
    \bar{D}(t)=D(t)/\prod_{s\in S'}(t-s) \quad \text{and} \quad \bar{D}_i(t)=D_i(t)/\prod_{s\in S'}(t-s)
    \]
    satisfy $q_i(t)=\bar{D}_i(t)/\bar{D}(t)$ for $t\notin S$, and $\deg\bar{D}_i\le\deg\bar{D}=m-|S'|$.
\end{claim}

\begin{proof}
    Since $t_0 \in S'$, the columns of $M(t_0)$ are linearly dependent and $\vec{r}(t_0) \in \textrm{Im}(M(t_0))$, so the columns of $M_i(t_0)$ are also linearly dependent. Thus $D(t_0) = D_i(t_0) = 0$, and therefore $t-t_0$ is a factor of both polynomials. Moreover, since $M(t)=2tI_m+(g_{i,k})_{1\le i,k\le m}$, the polynomial $D(t)$ has degree exactly $m$, with leading coefficient $2^m\neq 0$, while every entry of $M_i(t)$ has degree at most $1$, so $\deg D_i\le m$.
\end{proof}

Now we are ready to prove the following lemma:
\begin{lemma}
    There are at most $2+2m-|S'|$ possible trace values of elements in the zero set of a monic quadratic general polynomial $f(x)$.
\end{lemma}
\begin{proof}
    Note that in the expansion \eqref{eq:el-basis-expansion}, the value of $t$ for a root $q$ of $f(x)$ uniquely determines its trace, and vice versa. 
    Our discussion splits into three cases. If $t \in S \setminus S'$, then \eqref{eq:E} has no solution, so $f$ has no root $q$ with $\tr(q)/2 = t$. If $t \notin S$, then there is a potential unique solution $q_i(t) = D_i(t)/D(t)$ coming from equations \eqref{eq:E_1}--\eqref{eq:E_m}; by Claim~\ref{compatible}, $q_i(t)=\bar{D}_i(t)/\bar{D}(t)$. When the values of $q_1,\ldots,q_m$ are substituted into \eqref{eq:E_0} we get:
    \[t^2+\sum\limits_{i=1}^{m}\alpha_i\left(\frac{\bar{D}_i(t)}{\bar{D}(t)}\right)^2+G_0\left(t,\frac{\bar{D}_1(t)}{\bar{D}(t)},\dots,\frac{\bar{D}_m(t)}{\bar{D}(t)}\right)+r_0=0.\]
    Multiplying both sides by $\bar{D}(t)^2$ (which is nonzero because $t \notin S$) yields: 
    \[(t\bar{D}(t))^2+\sum\limits_{i=1}^{m}\alpha_i\bar{D}_i(t)^2+G_0\left(t\bar{D}(t)^2,\bar{D}_1(t)\bar{D}(t),\dots,\bar{D}_m(t)\bar{D}(t)\right)+r_0\bar{D}(t)^2=0.\]
      The left-hand side of the new equation is a nonzero polynomial in $t$ of degree at most $2+2(m-|S'|)$: indeed, $(t\bar{D}(t))^2$ has degree $2+2\deg\bar{D}$, while all other terms have degree at most $1+2\deg\bar{D}$ since $\deg\bar{D}_i\le\deg\bar{D}$, and $\deg\bar{D}=m-|S'|$ by Claim~\ref{compatible}.
    So there are at most $2+2(m-|S'|)$ possible trace values when $t \notin S$. Combining with the case $t \in S'$, we get that there are at most $2+2(m-|S'|)+|S'| = 2+2m-|S'|$ possible trace values for the zero set of $f(x)$.
\end{proof}

The lemma sheds light on the geometry of the zero set. Over an arbitrary field $F$, we call a root $q$ of the monic quadratic general polynomial $f$  \emph{isolated} if it is the only root of $f$ with trace $\tr(q)$. If $q$ is a root with $t=\tr(q)/2\notin S'$, then $t\notin S$ (as \eqref{eq:E} has no solution for $t\in S\setminus S'$), so $q$ is uniquely determined by $t$ via Cramer's rule; hence $q$ is isolated, and there are at most $2+2(m-|S'|)$ such roots. When $F=\mathbb{R}$, isolated roots are also isolated in the topological sense, since by the lemma the zero set of $f$ has only finitely many trace values. This leaves the case when $t \in S'$ to be studied.

   \begin{lemma}
   Let $f(x)$ be a monic quadratic general polynomial and let $t \in S'$. Then the roots $q$ of $f$ with $\tr(q)/2=t$ form the intersection of the quadric \eqref{eq:E_0} with a $k$-dimensional affine subspace of $F^{m+1}$, where $1 \leq k \leq m$.
   \end{lemma}
\begin{proof}
    For a fixed $t \in S'$, the solutions of \eqref{eq:E} with $\tr(q)/2 = t$ form a $k$-dimensional affine subspace, where $1 \leq k \leq m$. These solutions are zeros of $f(x)$ if they also satisfy the quadric \eqref{eq:E_0}.
\end{proof}

We now summarize everything discussed so far in the following theorem:
\begin{theorem} \label{geom-of-zeros}
    Given a monic quadratic general polynomial $f(x)$, the set of solutions to $f(q)=0$ consists of a finite set of isolated points and the union of intersections of the quadric \eqref{eq:E_0} with at most $m$ affine subspaces of positive dimension. Together, the total number of isolated points and affine subspaces does not exceed $2m+2$. 
\end{theorem}
\begin{proof}
    It is a consequence of the two previous lemmas and the fact that $0 \leq |S| \leq m$.
\end{proof}


\subsection{Locally-complex Cayley--Dickson algebras}
We now consider the case of locally-complex Cayley--Dickson algebras. Recall that this is the sequence of algebras $A_n$ over $\mathbb{R}$, where for each $n\ge 1$ we have $\ell_n^2 = -1$. This implies that for each basis element $e_i, i \ge 1$ we have $\alpha_i = e_i^2=-1$. Moreover, the quadric \eqref{eq:E_0} is a spherical equation in the variables $q_1,\dots,q_m$:

\begin{equation} \label{eq:E0R}
 \left(q_1-\frac{g_{0,1}}{2}\right)^2+\dots+\left(q_m-\frac{g_{0,m}}{2}\right)^2 - \left(t^2+g_{0,0}t+\frac{1}{4}(g_{0,1}^2+\dots+g_{0,m}^2)+r_0\right)=0   \tag{$E_0^\mathbb{R}$}
\end{equation}

Theorem~\ref{geom-of-zeros} is restated as follows:
\begin{theorem}
    Given a monic quadratic general polynomial over a locally-complex Cayley--Dickson~algebra, the set of zeros consists of a finite set of isolated points and at most $m$ spheres, not necessarily of the same dimension. Together, the number of spheres and isolated points does not exceed $2m+2$.
\end{theorem}
\begin{proof}
Fix $t\in S'$ and let $L_t\subseteq\mathbb{R}^m$ be the set of solutions $\vec{v} = (v_1, \dots, v_m)^T$ of \eqref{eq:E}; it is an affine subspace of dimension $k$, where $1\le k\le m$. Note that a vector $\vec{v} \in L_t$ corresponds to a root $q = t + \sum_{i=1}^m v_i e_i$ of $f$ if and only if it satisfies \eqref{eq:E0R}. Write $\vec c=\frac12(g_{0,1},\dots,g_{0,m})^T$ and $\rho(t)=t^2+g_{0,0}t+r_0+\frac14\sum_{i=1}^m g_{0,i}^2$, so that \eqref{eq:E0R} reads $\lVert\vec v-\vec c\rVert^2=\rho(t)$. Let $\vec c_t$ be the orthogonal projection of $\vec c$ onto $L_t$, and let $\delta_t=\lVert\vec c-\vec c_t\rVert$. For $\vec v\in L_t$ we have $\lVert\vec v-\vec c\rVert^2=\lVert\vec v-\vec c_t\rVert^2+\delta_t^2$, so the zeros of $f$ with $\tr(q)/2=t$ correspond to the points $\vec v\in L_t$ with $\lVert\vec v-\vec c_t\rVert^2=\rho(t)-\delta_t^2$. Hence this set is:
\begin{itemize}
\item empty, when $\rho(t)<\delta_t^2$;
\item the single point $\vec c_t$, when $\rho(t)=\delta_t^2$;
\item a $(k-1)$-dimensional sphere in $L_t$ with centre $\vec c_t$ and radius $\sqrt{\rho(t)-\delta_t^2}$, when $\rho(t)>\delta_t^2$ (for $k=1$, a pair of points).
\end{itemize}
In particular, if $\rho(t)<0$ the sphere \eqref{eq:E0R} is empty, and if $\rho(t)=0$ it reduces to the point $\vec c$. A zero arising in the second case is the only zero of $f$ with its trace, hence isolated. Each $t\in S'$ thus contributes at most one sphere or one isolated point, so the bound of Theorem~\ref{geom-of-zeros} still applies.
\end{proof}

We can use the methods of this section to extend the fundamental theorem of algebra to locally-complex Cayley--Dickson algebras for the special case of monic quadratic general polynomials.

\begin{theorem} \label{thm:real-quadratic}
    Every monic quadratic general polynomial $f(x)$ over a locally-complex Cayley--Dickson~algebra $A_n$ with $n\ge 1$ has a zero.
\end{theorem}
\begin{proof}
    Suppose by way of contradiction that there is no $q\in A_n$ with $f(q)=0$. Write $\vec{v}(t) = (q_1(t),\dots,q_m(t))^T$ as the general solution formula to \eqref{eq:E} given by Cramer's rule. Recall that for $t\notin S$ we have $q_i(t)=D_i(t)/D(t)$, where $\deg{D_i(t)}\le\deg{D(t)}$, so that $\lim_{t\to+\infty} \vec{v}(t) = \vec{v}_{\infty}$ for some $\vec{v}_{\infty} \in \mathbb{R}^m$.
    
    Denote by $\mathcal{H}(q_1,\dots,q_m,t)$ the left-hand side of \eqref{eq:E0R}. Since $M(t) = 2 t I_m + (g_{i,k})_{_{1\le i,k\le m}}$, the polynomial $D(t)$ has degree exactly $m = 2^n-1$, with leading coefficient $2^m$. Since $n \ge 1$, $m$ is odd, so $D(t)$ has a real root and hence $S\neq\emptyset$. As $S$ is also finite ($|S|\le m$), $t_0 := \max(S)$ exists. Notice that  
    \[ \lim_{t\to+\infty} \mathcal{H}(\vec{v}(t),t) = -\infty,\]
    as $\lim_{t\to+\infty} \vec{v}(t) = \vec{v}_{\infty}$ and the rest of the expression is dominated by $-t^2$. Also, $\mathcal{H}(\vec{v}(t),t)$ is continuous in the interval $(t_0,+\infty)$, so $\mathcal{H}(\vec{v}(t),t) < 0$ for $t \in (t_0,+\infty)$, because otherwise there exists $s \in (t_0,+\infty)$ such that $\mathcal{H}(\vec{v}(s),s) = 0$, but then $(s,\vec{v}(s))$ solves \eqref{eq:E_0}--\eqref{eq:E_m}, so $s+\sum_{i=1}^m q_i(s)e_i$ is a zero of $f$, a contradiction.

Now consider $\lim_{t\to t_0^{+}} \vec{v}(t)$. Since each $q_i(t)=D_i(t)/D(t)$ is a rational function, the one-sided limit $\lim_{t\to t_0^{+}} q_i(t)$ exists in $\mathbb{R}\cup\{\pm\infty\}$. If the limit of each coordinate is finite, i.e., if $\lim_{t\to t_0^{+}} \vec{v}(t) = \vec{p}_{t_0} \in \mathbb{R}^m$, then $\vec{p}_{t_0}$ is a solution to the matrix equation \eqref{eq:E} at $t = t_0$. Since $t_0 \in S$, it cannot be the only solution, so the solutions to \eqref{eq:E} at $t=t_0$ form an affine subspace of positive dimension passing through $\vec{p}_{t_0}$. Moreover, by continuity, one has $\mathcal{H}(\vec{p}_{t_0},t_0)\le 0$, and equality is impossible, since otherwise $t_0+\sum_{i=1}^m p_ie_i$, where $\vec{p}_{t_0}=(p_1,\dots,p_m)^T$, would be a zero of $f$. Hence $\mathcal{H}(\vec{p}_{t_0},t_0)<0$, i.e., $\vec{p}_{t_0}$ lies in the interior of the sphere defined in \eqref{eq:E0R}. Since every line through an interior point of a ball meets its boundary sphere, the affine subspace of solutions to \eqref{eq:E} at $t=t_0$ meets the sphere \eqref{eq:E0R}, and the points of this intersection are zeros of $f$, a contradiction.
    
    Thus, it must be that $\lim_{t\to t_0^{+}} q_i(t) = \pm \infty$ for some $i \in \set{1,\dots,m}$, in which case $\lim_{t\to t_0^{+}} \mathcal{H}(\vec{v}(t),t) = +\infty$, which is a contradiction to $\mathcal{H}(\vec{v}(t),t) < 0$ for $t \in (t_0,+\infty)$. The theorem is thus proven.
\end{proof}

\begin{remark} \label{rem:n-zero}
The assumption $n\ge 1$ in Theorem~\ref{thm:real-quadratic} cannot be dropped. For $n=0$, i.e., $A_0=\mathbb{R}$, we have $m=0$, so $D(t)\equiv 1$ and $S=\emptyset$; thus $t_0$ does not exist and the argument above breaks down. Indeed, the conclusion fails in this case: $f(x)=x^2+1$ has no zero in $\mathbb{R}$, and correspondingly $\mathcal{H}(t)=-(t^2+1)<0$ for all $t\in\mathbb{R}$.
\end{remark}

\section{Right eigenvalues of \texorpdfstring{$2\times 2$}{2x2} octonion matrices} \label{sec:right-eigenvalues}
 \label{sec:eigenvalues}

Given a square matrix $A$ with entries in an algebra $\mathcal{A}$, not necessarily associative, we say $\lambda\in\mathcal{A}$ is a right eigenvalue of $A$ if there exists a nonzero solution $\vec{v}\in\mathcal{A}^n$ to the equation $A\vec{v}=\vec{v}\lambda$, and it is a left eigenvalue of $A$ if there exists a nonzero solution $\vec{v}\in\mathcal{A}^n$ to the equation $A\vec{v}=\lambda \vec{v}$. As in classical linear algebra, $\lambda$ is a left eigenvalue if and only if the matrix $\lambda I - A$ is singular (but this is not true for \emph{right} eigenvalues). 

Let $\sigma_{L,\mathcal{A}}(A)$ and $\sigma_{R,\mathcal{A}}(A)$ denote the sets of left and right eigenvalues of the matrix $A$ over the algebra $\mathcal{A}$, respectively. (We omit the subscript $\mathcal{A}$ when the specific algebra is clear from context.)

The importance of right eigenvalues becomes apparent when discussing diagonalization of matrices, or at least the well-known equation
\[AU=UD,\]
where $D$ is the diagonal matrix with entries $\lambda_1,\ldots,\lambda_n$ and the columns $\vec{u}_1,\ldots,\vec{u}_n$ of $U$ are nonzero. Comparing the $j$-th columns gives $A\vec{u}_j=\vec{u}_j\lambda_j$, so $\lambda_1,\ldots,\lambda_n\in\sigma_{R,\mathcal{A}}(A)$.

In \cite{ChapmanVishkautsan:2025}, it was shown that every $2\times 2$ matrix over the real octonions $\mathbb{O}$ admits a left eigenvalue. In this section, we extend this result to right eigenvalues as well.

\begin{theorem}
     Every $A \in M_2(\mathbb{O})$ admits a right eigenvalue, where $\mathbb{O}$ is the real octonion division algebra.
\end{theorem}

\begin{proof}
Write $A=\left(\begin{array}{cc}
    a & b \\
    c & d
\end{array}\right)$.
    When $b=0$, $d$ is clearly a right eigenvalue, with eigenvector $\vec{v}=\left( \begin{array}{cc}
         0  \\
         1 
    \end{array}\right).$ So, assume $b \neq 0$.
    It is enough to find $x$ and $\lambda$ in $\mathbb{O}$ for which 
    \[A \left( \begin{array}{cc}
         1  \\
         x 
    \end{array} \right)=\left( \begin{array}{cc}
         1  \\
         x 
    \end{array} \right)\lambda.
    \]
The requirement above translates to the system of equations
\[ \begin{array}{lcr}
   a+bx=\lambda   \\
    c+dx=x\lambda.
\end{array}\]
Substituting the first equation in the second gives
\[c+dx=x(a+bx), \]
which gives
\begin{equation} \label{eq:right-eigenvalues}
  0=x(bx)+xa-dx-c.  
\end{equation}
Multiplying both sides by $b$ from the left yields
\[0=b(x(bx))+b(xa)-b(dx)-bc.\]
By Artin's theorem, the subalgebra generated by $\{b,x\}$ is associative. Therefore, it must be $b(x(bx))=(bx)(bx)$. Setting $y=bx$ (recall $b\ne 0$, so  $x=b^{-1}y$), the equation becomes:
\[0=y^2+b((b^{-1}y)a)-b(d(b^{-1}y))-bc.\]
The latter has a solution $y\in\mathbb{O}$ by Theorem~\ref{thm:real-quadratic}, and all its solutions can be determined by the method of Section~\ref{sec:monic}. The choice of $x=b^{-1} y$ and $\lambda=a+bx$ provides a solution to the original system. This $\lambda$ is thus a right eigenvalue of $A$. 
\end{proof}

\begin{remark}
The left multiplication by $b$ and change of variables in the proof are unnecessary to prove the \emph{existence} of a right eigenvalue. Indeed, the general polynomial equation \eqref{eq:right-eigenvalues} already satisfies the hypotheses for the existence of a root in Jou's theorem \cite{jou1950}. However, they allow us to apply Theorem~\ref{thm:real-quadratic} and use the method in Section~\ref{sec:monic}.
\end{remark}

\begin{example}
Denote by $i,j,\ell$ the standard generators of $\mathbb{O}$ (these are $\ell_1,\ell_2,\ell_3$, respectively, in the Cayley--Dickson process). 
    Consider the matrix 
    \[A=\left( \begin{array}{cc}
         1 & i \\
         1+2i+j & 1+\ell
    \end{array}\right).\]
    The equation obtained by the proof of the previous theorem is 
    \[0=y^2+\ell y+2-i-ij,\]
    which has a solution $y=(i+ij)\ell$. Therefore, $x=-iy=((i+ij)(-i))\ell=(1-j)\ell$, and so $\lambda=1+ix=1+(i+ij)\ell$ is a right eigenvalue of $A$.
\end{example}

\section{Eigenvalues of triangular matrices} \label{sec:triangular}

In this section we discuss the left and right spectra of triangular matrices. For a square matrix $B$, we denote by $\mathrm{diag}(B)$ the set of its diagonal entries.
Let $F$ denote an arbitrary field of any characteristic. By an octonion algebra over $F$ we mean an $8$-dimensional composition algebra over $F$. When $\charac{F}\neq 2$, these are precisely the algebras $A_3$ of Section~\ref{sec:cd}. In characteristic $2$ they are obtained by the Cayley--Dickson doubling process starting from a quadratic \etale{} algebra over $F$ (see e.g.\ \cite[Section 2.1]{Chapman-et-al:2023}). 

\begin{proposition} \label{prop:triangular-right}
    Let $B\in M_n(O)$ be an upper (lower) triangular matrix where $O$ is an octonion division algebra over any field $F$ (of arbitrary characteristic). Then 
    \[
    \sigma_R(B)=\bigcup_{d\in\mathrm{diag}(B)}[d],
    \]
    where $[d]$ is the conjugacy class of $d$ in $O$, i.e., $[d]=\{qdq^{-1} : 0\neq q\in O\}$ (conjugacy is an equivalence relation on $O$; see \cite[Remark~3.1]{Chapman:2020a}).
\end{proposition}

\begin{proof}
  We treat the upper triangular case; the lower triangular case is analogous, with the order of the indices reversed.
  
    Let 
    \[ B = 
\begin{pmatrix}
a_{11} & a_{12} & a_{13} & \cdots & a_{1n} \\
0      & a_{22} & a_{23} & \cdots & a_{2n} \\
0      & 0      & a_{33} & \cdots & a_{3n} \\
\vdots & \vdots & \vdots & \ddots & \vdots \\
0      & 0      & 0      & \cdots & a_{nn}
\end{pmatrix},
\]
and let $\vec{v} = (c_1,\ldots,c_n)^T\in O^n$ be an eigenvector associated with a right eigenvalue $\lambda$. 
Let $k$ be the largest index such that $c_k\neq 0$.
Then the $k$-th row of the equality $B\vec{v}=\vec{v}\lambda$ is $a_{kk}c_k = c_k\lambda$. We thus get $\lambda = c_k^{-1}a_{kk}c_k$ (parentheses are unnecessary due to alternativity), and therefore $\lambda \in [a_{kk}]$. We have thus proved that $\sigma_R(B)\subseteq\bigcup_{d\in\mathrm{diag}(B)}[d]$.

In the other direction, let $1\le m\le n$, and let $1\le k \le m$ be the minimal integer such that $a_{kk}\in [a_{mm}]$. Set $\lambda = q^{-1}a_{kk}q$ for some $q\ne 0$ (this is a general element in $[a_{mm}]$ as well). We define the eigenvector $\vec{v} = (c_1,\ldots,c_n)^T\in O^n$, by first setting $c_k = q$ and $c_i = 0$ for $k<i\le n$; then we define $c_i, 1\le i<k$ by back substitution: in row $k-1$ we get the following equation 
\begin{equation} \label{triangular-right-eigenvalue-eq1}
    a_{k-1,k-1}c_{k-1} - c_{k-1}\lambda = -a_{k-1,k}q.
\end{equation}
This is a Sylvester-type equation of the form $ax-xb=c$, where $x=c_{k-1}$. The function $f(x)=ax-xb$ induces a linear transformation from $F^8$ to itself. The transformation is singular if and only if there exists $0\ne s\in O$ such that $f(s)=0$, and this is equivalent to $b=s^{-1}as$, i.e., if and only if $b\in [a]$. By the choice of $k$, we have $a_{k-1,k-1}\notin [a_{kk}]=[\lambda]$; this proves that the linear transformation is non-singular, and there exists a solution $c_{k-1}$ for \eqref{triangular-right-eigenvalue-eq1}.
A general equation for row $1\le i<k$ takes the form 
\[a_{ii}c_i - c_i\lambda = -a_{i,i+1}c_{i+1}-\ldots-a_{i,k-1}c_{k-1} - a_{i,k}q ,\] 
and allows us to proceed by back substitution, proving the existence of an eigenvector associated with $\lambda$. This proves $\sigma_R(B)\supseteq\bigcup_{d\in\mathrm{diag}(B)}[d]$.
\end{proof}

We take the opportunity to correct the proof of \cite[Lemma~3.2]{ChapmanVishkautsan:2025}, which states that the left spectrum of a triangular matrix over a (not necessarily associative) division algebra consists of its diagonal entries; here a division algebra is a unital algebra $D$ (with $1\ne 0$) in which the equations $ax=b$ and $ya=b$ have unique solutions for all $0\neq a\in D$ and $b\in D$ (see \cite[Section II.2]{Schafer:1966}). The proof given there of the inclusion $\sigma_L(B)\subseteq\mathrm{diag}(B)$ is correct, but the argument for the reverse inclusion is not: it claims that the $i$-th standard basis vector is an eigenvector for the $i$-th diagonal entry, which is false in general. We correct it below by a standard argument of back substitution.

\begin{proposition} \label{prop:triangular-left}
    Let $B\in M_n(D)$ be an upper (lower) triangular matrix where $D$ is a (not necessarily associative) division algebra over a field $F$. Then $\sigma_L(B)=\mathrm{diag}(B).$
\end{proposition}

\begin{proof}  We treat the upper triangular case; the lower triangular case is analogous, with the order of the indices reversed.

    Let 
    \[ B = 
\begin{pmatrix}
a_{11} & a_{12} & a_{13} & \cdots & a_{1n} \\
0      & a_{22} & a_{23} & \cdots & a_{2n} \\
0      & 0      & a_{33} & \cdots & a_{3n} \\
\vdots & \vdots & \vdots & \ddots & \vdots \\
0      & 0      & 0      & \cdots & a_{nn}
\end{pmatrix},
\]
and let $\vec{v} = (c_1,\ldots,c_n)^T\in D^n$ be an eigenvector associated with a left eigenvalue $\lambda$. 
Let $k$ be the largest index such that $c_k\neq 0$.
Then the $k$-th row of the equality $B\vec{v}=\lambda \vec{v}$ is $a_{kk}c_k = \lambda c_k$, i.e., $(a_{kk}-\lambda)c_k=0$. Since $c_k\neq 0$ and $D$ is a division algebra, we obtain $\lambda = a_{kk}$, proving $\sigma_L(B)\subseteq\mathrm{diag}(B).$

In the other direction, given $\lambda = a_{mm},$ for some $1\le m\le n$, set $k$ to be the minimal integer between $1$ and $m$ such that $a_{kk} = a_{mm}$. We need to show that there exists an eigenvector $\vec{v}= (c_1,\ldots,c_n)^T\in D^n$ associated with $\lambda$. We set $c_k = 1$ and $c_i=0$ for $k<i\le n,$ and define $c_i$ for $1\le i < k$ by back substitution as the unique solution of the $i$-th row equation
$(a_{ii}-\lambda)c_i = -a_{i,i+1}c_{i+1}-\ldots-a_{i,k-1}c_{k-1}-a_{ik}$, which exists since $a_{ii}\neq\lambda$ by the choice of $k$. Therefore, by construction,  $B\vec{v}=\lambda \vec{v}$, and thus we have proved that $\sigma_L(B)\supseteq\mathrm{diag}(B).$
\end{proof}


\end{document}